\documentclass[]{interact}
\usepackage{color}
\usepackage{epstopdf}
\usepackage{caption}
\usepackage{cases}
\usepackage{subfigure}
\usepackage{graphics,graphicx}
\usepackage{algorithm,algorithmic}

\usepackage{caption}
\usepackage[colorlinks,
            linkcolor=red,
            anchorcolor=blue,
            citecolor=blue
            ]{hyperref}

\usepackage{cleveref}
\usepackage[numbers,sort&compress]{natbib}
\bibpunct[, ]{[}{]}{,}{n}{,}{,}
\makeatletter
\def\NAT@def@citea{\def\@citea{\NAT@separator}}
\makeatother

\theoremstyle{plain}
\newtheorem{theorem}{Theorem}[section]
\newtheorem{lemma}[theorem]{Lemma}

\newtheorem{proposition}[theorem]{Proposition}
\theoremstyle{definition}

\newtheorem{example}[theorem]{Example}

\theoremstyle{remark}
\newtheorem{remark}{Remark}

\begin{document}
\title{SOR-like iteration and FPI are consistent when they are equipped with certain optimal iterative parameters}

\author{
\name{Jiayu Liu\textsuperscript{a}\thanks{Email address: 1977078576@qq.com.} and Tingting Luo\textsuperscript{a}\thanks{Email address: 610592494@qq.com.} and Cairong Chen\textsuperscript{a}\thanks{Corresponding author. Email address: cairongchen@fjnu.edu.cn.} and Deren Han\textsuperscript{b}\thanks{Email address: handr@buaa.eud.cn.}}
\affil{\textsuperscript{a}School of Mathematics and Statistics \& Key Laboratory of Analytical Mathematics and Applications (Ministry of Education) \& Fujian Provincial Key Laboratory of Statistics and Artificial Intelligence, Fujian Normal University, Fuzhou, 350117, P.R. China}
\affil{\textsuperscript{b}LMIB of the Ministry of Education, School of Mathematical Sciences, Beihang University, Beijing, 100191, P.R. China}
}
\maketitle

\begin{abstract}
Two common methods for solving absolute value equations (AVE) are SOR-like iteration method and fixed point iteration (FPI) method. In this paper, novel convergence analysis, which result wider convergence range, of the SOR-like iteration and the FPI are given. Based on the new analysis, a new optimal iterative parameter with a analytical form is obtained for the SOR-like iteration. In addition, an optimal iterative parameter with a analytical form is also obtained for FPI. Surprisingly, the SOR-like iteration and the FPI are the same whenever they are equipped with our optimal iterative parameters. As a by product, we give two new constructive proof for a well known sufficient condition such that AVE has a unique solution for any right hand side. Numerical results demonstrate our claims.
\end{abstract}


\section{Introduction}\label{sec:intro}
We consider  absolute value equations (AVE) of the form
\begin{equation}\label{eq:ave}
Ax - | x | = b,
\end{equation}
where $A\in\mathbb{R}^{n\times n}$, $b\in\mathbb{R}^n$, and $|x|\in\mathbb{R}^n$ denotes the entrywise absolute value of the unknown vector $x\in\mathbb{R}^n$. AVE \eqref{eq:ave} can be regarded as a special case of the general absolute value equation (GAVE)
\begin{equation}\label{eq:gave}
Cx - D|x| = e,
\end{equation}
where $C,D\in\mathbb{R}^{m\times n}$ and $e\in \mathbb{R}^m$. It was known that determining the existence of a solution to the general GAVE is NP-hard \cite{mang2007a}, and if it has a solution, determining whether the GAVE has a unique solution or multiple solutions is NP-complete \cite{prok2009}. For further investigation on GAVE, one can see \cite{hlad2018,love2013,mezz2020,rohn2009a,rohf2014,wush2021}. Over the past two decades, AVE \eqref{eq:ave} has garnered significant attention in the community of numerical optimization since it is closely related to many mathematical programming problems, which include linear complementarity problems (LCP) \cite{huhu2010,mang2014,mame2006,prok2009}. In addition, AVE \eqref{eq:ave} also arises from the characterization of certain solutions to the system of linear interval equations \cite{rohn1989,rohn2004}. Recently, a transform function based on the underdetermined GAVE~\eqref{eq:gave} is used to improve the security of the cancellable biometric system \cite{dnhk2023}.

Given these diverse applications and theoretical significance, developing efficient numerical methods for solving AVE \eqref{eq:ave} remains as an active research topic. In recent years, there has been numerous algorithms for solving AVE \eqref{eq:ave}. For example, Newton-type iteration methods \cite{mang2009a,lilw2018,bcfp2016,wacc2019}, iterative methods based on matrix splitting \cite{lild2022,kema2017,edhs2017}, concave minimization approaches \cite{mang2007b,zahl2021}, methods based on neurodynamic models \cite{cyyh2021,yzch2024}, and others; see, e.g.,  \cite{ke2020,alct2023,chyh2023,xiqh2024,soso2023,bcfp2016,maer2018,abhm2018,sayc2018,tazh2019}.

The goal of this paper is to revisit the convergence conditions and optimal iterative parameters for two of the above-mentioned algorithms, i.e., the SOR-like iteration method \cite{kema2017} and the fixed point iteration (FPI) method \cite{ke2020}. In the following, we briefly review these two methods.

Let $y = |x|$, AVE~\eqref{eq:ave} is equivalent to
\begin{equation}\label{eq:ave-eq}
\mathcal{A}z := \begin{bmatrix} A &-I\\ -\mathcal{D}(x) & I\end{bmatrix} \begin{bmatrix} x\\ y\end{bmatrix} = \begin{bmatrix} b\\ 0\end{bmatrix} := c,
\end{equation}
where $\mathcal{D}(x) = {\rm diag}({\rm sign}(x))$. By splitting
$$
\omega\mathcal{A} = \begin{bmatrix} A &0\\ -\omega \mathcal{D}(x) & I\end{bmatrix} - \begin{bmatrix} (1-\omega)A &\omega I\\0 & (1-\omega)I\end{bmatrix}
$$
with $\omega> 0$ is the iterative parameter, Ke and Ma \cite{kema2017} proposed the following SOR-like iteration for solving AVE~\eqref{eq:ave}:
\begin{equation*}
\begin{bmatrix} A &0\\ -\omega \mathcal{D}(x^{(k+1)}) & I\end{bmatrix} \begin{bmatrix} x^{(k+1)}\\ y^{(k+1)}\end{bmatrix} = \begin{bmatrix} (1-\omega)A &\omega I\\0 & (1-\omega)I\end{bmatrix}\begin{bmatrix} x^{(k)}\\ y^{(k)}\end{bmatrix} + \begin{bmatrix} \omega b\\ 0\end{bmatrix}.
\end{equation*}
The SOR-like iteration method is described in \Cref{alg:SOR}.
\begin{algorithm}[htp]
\caption{SOR-like iteration method for solving AVE \eqref{eq:ave} \cite{kema2017}.}\label{alg:SOR}
Let $A\in \mathbb{R}^{n\times n}$ be a nonsingular matrix and $b\in \mathbb{R}^{n}$. Given the initial vectors $x^{(0)}\in \mathbb{R}^{n}$ and $y^{(0)}\in \mathbb{R}^{n}$, for $k=0,1,2,\cdots$ until the iteration sequence $\left\{(x^{(k)},y^{(k)})\right\}_{k=0}^\infty$ is convergent, compute
\begin{eqnarray}\label{eq:sor}
\begin{cases}
x^{(k+1)}=(1-\omega)x^{(k)}+\omega A^{-1}(y^{(k)}+b),\\
y^{(k+1)}=(1-\omega)y^{(k)}+\omega |x^{(k+1)}|,
\end{cases}
\end{eqnarray}
where  $\omega > 0$ is the iterative parameter.
\end{algorithm}

Hereafter, based on \eqref{eq:ave-eq} again, Ke \cite{ke2020} proposed the following FPI method (see \Cref{alg:FPI}) for solving AVE~\eqref{eq:ave}.
\begin{algorithm}[htp]
\caption{FPI method for solving AVE \eqref{eq:ave} \cite{ke2020}}\label{alg:FPI}
Let $A\in \mathbb{R}^{n\times n}$ be a nonsingular matrix and $b\in \mathbb{R}^{n}$. Given the initial vectors $x^{(0)}\in \mathbb{R}^{n}$ and $y^{(0)}\in \mathbb{R}^{n}$, for $k=0,1,2,\cdots$ until the iteration sequence $\left\{(x^{(k)},y^{(k)})\right\}_{k=0}^\infty$ is convergent, compute
 \begin{eqnarray}\label{eq:fpi}
\begin{cases}
x^{(k+1)}=A^{-1}(y^{(k)}+b),\\
y^{(k+1)}=(1-\tau)y^{(k)}+\tau |x^{(k+1)}|,
\end{cases}
\end{eqnarray}
where $\tau>0$ is the iterative parameter.
\end{algorithm}

Let $(x_*, y_*)$ be the solution pair of the nonlinear equation \eqref{eq:ave-eq} and define
$$
e_k^x = x_* - x^{(k)}, e_k^y = y_* - y^{(k)}.
$$
Then we can review the following results.

For the SOR-like iteration method, Ke and Ma obtain the following theorem.
\begin{theorem}[{\cite[Theorem 2.1]{kema2017}}]\label{thm:kema}
Assume that $A \in \mathbb{R}^{n\times n}$ is a nonsingular matrix and $b\in \mathbb{R}^{n}$. Denote
$$ \nu=\|A^{-1}\|_2, \quad a=|1-\omega|\quad \text{and}\quad d=\omega^2\nu. $$
For the sequence $\{(x^{(k)},y^{(k)})\}$ generated by \eqref{eq:sor}, if
\begin{equation}\label{eq:cond1}
0<\omega< 2  \qquad \text{and} \qquad  a^4-3a^2 -2ad- 2d^2 +1 >0,
\end{equation}
the following inequality
\begin{equation*}
\| |(e_{k+1}^x,e_{k+1}^y)| \|_{\omega} < \| |(e_k^x,e_k^y) |\|_{\omega}
\end{equation*}
holds for $ k=0,1,2,\cdots $. Here the norm $\| |\cdot|\|_{\omega}$ is defined by
$$\| |(e_k^x,e_k^y) |\|_{\omega}:=\sqrt {\|e_k^x \|_2^2+\omega ^{-2}\|e_k^y \|_2^2 }.$$
\end{theorem}

Recently, Chen et al. \cite{chyh2024} revisited the convergence condition \eqref{eq:cond1} of the SOR-like iteration method and determined the optimal iteration parameter which minimizes $\|T_{\nu}(\omega)\|_A$ with
$$T_\nu(\omega) = \begin{bmatrix} |1-\omega| & \omega\nu \\  \omega |1-\omega|  &  |1-\omega| +\omega^2\nu \end{bmatrix}$$
and $A = \begin{bmatrix} 1 & 0\\ 0 &\frac{1}{\omega^2}\end{bmatrix}$ such that
\begin{equation}\label{eq:errsor}
0\le \| (\|e_{k+1}^x\|_2,\|e_{k+1}^y\|_2) \|_A \le \|T_\nu(\omega) \|_A \cdot \| (\|e_k^x\|_2,\|e_k^y\|_2) \|_A.
\end{equation}
Here, $\|x\|_A = \sqrt{x^\top Ax}$ and $\|X\|_A = \|A^{\frac{1}{2}}XA^{-\frac{1}{2}}\|_2$.
From \eqref{eq:errsor}, for the sequence $\{(\|e_x^k\|_2, \|e^k_y\|_2)\}$, $\|T_{\nu}(\omega)\|_A$ is an upper bound of the linear convergence factor for the SOR-like iteration method in terms of the metric $\|\cdot \|_A$. However, the metric $\|\cdot \|_A$ is $\omega$-dependent and the resulting optimal iterative parameter doesn't have a analytical form (see \eqref{eq:opt}). This brings out an interesting question on finding an optimal iterative parameter with a analytical form. To this end, we reanalysis the convergence of the SOR-like iteration method without using the metric $\|\cdot \|_A$.

For the FPI method, Ke proposed the following theorem.

\begin{theorem}[{\cite[Theorem 2.1]{ke2020}}]\label{thm:kefpi} Assume that $A \in \mathbb{R}^{n\times n}$ is a nonsingular matrix and $b\in \mathbb{R}^{n}$. Denote
$$\nu=\|A^{-1}\|{_2}\quad \text{and}\quad E^{(k+1)}=\begin{bmatrix}\begin{array}{c} \|e_{k+1}^x\|_2\\ \|e_{k+1}^y\|_2\end{array}\end{bmatrix}.$$
For the sequence $\{(x^{(k)},y^{(k)})\}$ generated by \eqref{eq:fpi}, if
\begin{equation}\label{eq:cfpi}
0<\nu<  \frac{\sqrt{2}}{2} \qquad \text{and}  \qquad  \frac{1- \sqrt{1- \nu^2}}{1- \nu} < \tau <  \frac{1+\sqrt{1-\nu^2}}{1+\nu},
\end{equation}
$\|E^{(k+1)}\|_2< \|E^{(k)}\|_2$ for all $k=0,1,2,\cdots$.
\end{theorem}

For AVE~\eqref{eq:ave}, the following \Cref{pro:us} reveals a sufficient condition such that AVE~\eqref{eq:ave} has a unique solution for any $b \in \mathbb{R}^{n}$. However, in \eqref{eq:cfpi}, $\nu\in (0, \frac{\sqrt{2}}{2})$. There exists a gap between $(0, \frac{\sqrt{2}}{2})$ and $(0, 1)$. In order to theoretically fill this gap, Yu et al. \cite{yuch2022} modified the FPI by introducing an auxiliary matrix. However, the optimal iterative parameter of the FPI method is still lack in the literature. This motivates us to give a new convergence analysis of the FPI method which not only can fill the above-mentioned gap without modifying the original FPI but also can shine the light into determining the optimal iterative parameter.

\begin{proposition}[\cite{mame2006}]\label{pro:us}
Assume that $A \in \mathbb{R}^{n\times n}$ is invertible. If $\|A\|_2^{-1}<1$,   AVE~\eqref{eq:ave} has a unique solution for any $b \in \mathbb{R}^{n}$.
\end{proposition}

Generally, the SOR-like iteration \eqref{eq:sor} and the FPI \eqref{eq:fpi} are different from each other. Surprisingly, our analysis below investigates that the SOR-like iteration \eqref{eq:sor} and the FPI \eqref{eq:fpi} are the same whenever they are equipped with our optimal iterative parameters. Our work makes the following key contributions:

\begin{enumerate}
  \item For the SOR-like iteration method, new convergence result and optimal iteration parameter are given. The new convergence range is larger than the existing one and the new optimal iteration parameter has a analytical form.

  \item For the FPI method, new convergence result is given. Unlike \cite{yuch2022}, we theoretically fill the convergence gap without modifying the original method. Furthermore, we obtain the optimal iterative parameter.

  \item We discover that the SOR-like iteration and and the FPI are the same when they are equipped with our optimal iterative parameters.
\end{enumerate}

The rest of this paper is organized as follows: In \Cref{sec:Preliminaries}, we present preliminary results and essential lemmas that serve as the foundation for our subsequent analysis. In \Cref{sec:SOR} and \Cref{sec:FPI}, we establishes broader convergence domains and derives explicit expressions for optimal iteration parameters of the SOR-like iteration and FPI, respectively. Numerical results are given in \Cref{sec:ne}. Finally, some concluding remarks are given in \Cref{sec:conclusions}.

\textbf{Notation.} Let $\mathbb{R}^{n\times n}$ be the set of all $n\times n$ real matrices and $\mathbb{R}^n=\mathbb{R}^{n\times 1}$. $|U|\in\mathbb{R}^{m\times n}$ denote the componentwise absolute value of the matrix $U$. $I$ denotes the identity matrix with suitable dimensions. $\|U\|_2$ denotes the $2$-norm of $U\in\mathbb{R}^{m\times n}$ which is defined by the formula $\|U\|_2=\max\{\|Ux\|_2:x\in\mathbb{R}^n,\|x\|_2=1\}$, where $\|x\|_2$ is the $2$-norm of the vector $x$. $\rho(U)$ denotes the spectral radius of $U$. For $A \in \mathbb{R}^{n\times n}$, $\det (A)$ denotes its determinant. The sign of a real $r$ is defined by ${\rm sign}(r)=1$ if $r> 0$, $0$ if $r=0$ and $-1$ if $r<0$. For $x\in \mathbb{R}^n$, ${\rm diag}(x)$ represents a diagonal matrix with $x_i$ as its diagonal entries for every $i = 1,2,\ldots,n$.

\section{Preliminaries}\label{sec:Preliminaries}
In this section, we collect some basic results that will be used later.

\begin{lemma}[{\cite[Lemma 2.1]{youn1971}}]\label{lem:2.1}
Let $p$ and $q$ be real coefficients. Then both roots of the quadratic equation $x^2 - px + q = 0$ are less than one in modulus if and only if $|q|<1$ and $|p|<1+q$.
\end{lemma}

\begin{lemma}[{e.g., \cite[Theorem~1.10]{saad2003}}]\label{lem:2.4}
For~$U\in\mathbb{R}^{n\times n}$,~$\lim\limits_{k\rightarrow+\infty} U^k=0$ if and only if~$\rho(U)<1$.
\end{lemma}

\begin{lemma}[{e.g., \cite[Theorem~1.11]{saad2003}}]\label{lem:2.3}
For~$U\in\mathbb{R}^{n\times n}$, the series~$\sum\limits_{k=0}^\infty U^k$ converges if and only if~$\rho(U)<1$ and we have~$\sum\limits_{k=0}^\infty U^k=(I-U)^{-1}$ whenever it converges.
\end{lemma}

\section{New convergence and new optimal iterative parameter of SOR-like iteration}\label{sec:SOR}
In this section, we devote to giving new convergence analysis and deriving new optimal iterative parameter for the SOR-like iteration method.

\subsection{New convergence analysis}
In this subsection, we derive a new convergence condition for the SOR-like iteration method, which results a larger range of $\omega$ than that of \cite{chyh2024}. Concretely, we have the following theorem.

\begin{theorem}\label{thm:sor}
Let $A\in \mathbb{R}^{n\times n}$ be a nonsingular matrix and denote $\nu=\|A^{-1}\|_2$. If
\begin{equation}\label{eq:con-sor}
0<\nu<1 \quad \text{and}\quad 0<\omega<\frac{2 - 2\sqrt{\nu}}{1 - \nu},
\end{equation}
AVE \eqref{eq:ave} has a unique solution for any $b\in \mathbb{R}^n$ and the sequence~$\{(x^{(k)},y^{(k)})\}^\infty_{k=0}$ generated by~\eqref{eq:sor} globally linearly  converges to~$(x_{*}, y_{*}=|x_*|)$ with $x_{*}$ being the unique solution of AVE~\eqref{eq:ave}.
\end{theorem}

\begin{proof}
It follows from \eqref{eq:sor} that
\begin{eqnarray}\label{eq:sor'}
\begin{cases}
x^{(k)}=(1-\omega)x^{(k-1)}+\omega A^{-1}(y^{(k-1)}+b),\\
y^{(k)}=(1-\omega)y^{(k-1)}+\omega |x^{(k)}|.
\end{cases}
\end{eqnarray}
Subtracting~\eqref{eq:sor'} from~\eqref{eq:sor}, we have
\begin{eqnarray*}
\begin{cases}
x^{(k+1)}-x^{(k)}=(1-\omega)(x^{(k)}-x^{(k-1)})+\omega A^{-1}(y^{(k)}-y^{(k-1)}),\\
y^{(k+1)}-y^{(k)}=(1-\omega)(y^{(k)}-y^{(k-1)})+\omega (|x^{(k+1)}|-|x^{(k)}|),
\end{cases}
\end{eqnarray*}
from which and $\||x| - |y|\|_2 \le \|x - y\|_2$ that
\begin{eqnarray*}
\begin{cases}
\|x^{(k+1)}-x^{(k)}\|_2 \leq |1-\omega| \|x^{(k)}-x^{(k-1)}\|_2 +\omega \nu \|y^{(k)}-y^{(k-1)}\|_2,\\
\|y^{(k+1)}-y^{(k)}\|_2 \leq |1-\omega| \|y^{(k)}-y^{(k-1)}\|_2 +\omega \|x^{(k+1)}-x^{(k)}\|_2.
\end{cases}
\end{eqnarray*}
That is,
\begin{equation}\label{eq:sor*}
\begin{bmatrix}
  1 & 0 \\
  -\omega & 1
\end{bmatrix}
\begin{bmatrix}
  \|x^{(k+1)}-x^{(k)}\|_2 \\
  \|y^{(k+1)}-y^{(k)}\|_2
\end{bmatrix}
\leq
\begin{bmatrix}
  |1-\omega| & \omega\nu \\
  0 & |1-\omega|
\end{bmatrix}
\begin{bmatrix}
  \|x^{(k)}-x^{(k-1)}\|_2 \\
  \|y^{(k)}-y^{(k-1)}\|_2
\end{bmatrix}.
\end{equation}
Multiplying \eqref{eq:sor*} from left by the nonnegative matrix
$
\begin{bmatrix}
  1 & 0 \\
  \omega & 1
\end{bmatrix}
$,
we get
\begin{equation}\label{eq:W}
\begin{bmatrix}
  \|x^{(k+1)}-x^{(k)}\|_2 \\
  \|y^{(k+1)}-y^{(k)}\|_2
\end{bmatrix}
\leq W
\begin{bmatrix}
  \|x^{(k)}-x^{(k-1)}\|_2 \\
  \|y^{(k)}-y^{(k-1)}\|_2
\end{bmatrix}
\end{equation}
with
\begin{equation}\label{eq:w}
W=\begin{bmatrix}
    |1-\omega| & \omega\nu \\
    \omega |1-\omega| & \omega^2 \nu+|1-\omega|
  \end{bmatrix}\ge 0.
\end{equation}

For each $m \geq 1$, if $\rho(W)<1$, it follows from~\eqref{eq:W}, \eqref{eq:w}, \Cref{lem:2.4} and  \Cref{lem:2.3}  that
\begin{align*}
\left[\begin{array}{c}
 \|x^{(k+m)}-x^{(k)}\|_2 \\
 \|y^{(k+m)}-y^{(k)}\|_2
\end{array}\right]&=
\left[\begin{array}{c}
 \|\sum_{j=0}^{m-1}(x^{(k+j+1)}-x^{(k+j)})\|_2 \\
 \|\sum_{j=0}^{m-1}(y^{(k+j+1)}- y^{(k+j)})\|_2
\end{array}\right]
\leq
\sum_{j=0}^{m-1}
\left[\begin{array}{c}
 \|x^{(k+j+1)}-x^{(k+j)}\|_2 \\
 \|y^{(k+j+1)}- y^{(k+j)}\|_2
\end{array}\right]\nonumber\\
&\leq \sum_{j=0}^{\infty}W^{j+1}
\left[\begin{array}{c}
 \|x^{(k)}- x^{(k-1)}\|_2 \\
 \|y^{(k)}- y^{(k-1)}\|_2
\end{array}\right]
=(I-W)^{-1}W
\left[\begin{array}{c}
 \|x^{(k)}-x^{(k-1)}\|_2 \\
 \|y^{(k)}-y^{(k-1)}\|_2
\end{array}\right]\nonumber\\
&\leq (I-W)^{-1}W^k
\left[\begin{array}{c}
 \|x^{(1)}-x^{(0)}\|_2 \\
 \|y^{(1)}-y^{(0)}\|_2
\end{array}\right]
\rightarrow
\left[\begin{array}{c}
0\\
0
\end{array}\right]~~(\text{as}\quad k\rightarrow \infty).
\end{align*}
Hence, $\{x^{(k)}\}_{k=0}^{\infty}$ and~$\{y^{(k)}\}_{k=0}^{\infty}$ are Cauchy sequences and they are convergent in $\mathbb{R}^n$. Let $\lim\limits_{k\rightarrow\infty} x^{(k)} =x_{*}$ and $\lim\limits_{k\rightarrow\infty} y^{(k)} =y_{*}$, it follows from~\eqref{eq:sor} that
\begin{eqnarray*}
\begin{cases}
x_*=(1-\omega)x_*+\omega A^{-1}(y_*+b),\\
y_*=(1-\omega)y_*+\omega |x_*|,
\end{cases}
\end{eqnarray*}
from which and $\omega>0$ we have
\begin{eqnarray*}
\begin{cases}
Ax_{*}-y_*-b=0,\\
y_{*} = |x_*|.
\end{cases}
\end{eqnarray*}
Thus, $x_{*}$ is a solution to AVE~\eqref{eq:ave}.

Next, we turn to consider the conditions such that $\rho(W)<1$. Suppose that~$\lambda$ is an eigenvalue of~$W$, and then
\begin{eqnarray*}
\det (\lambda I-W)=\det\left(
\begin{bmatrix}
  \lambda-|1-\omega| & -\omega\nu \\
  -\omega|1-\omega| & \lambda-(\omega^2 \nu+|1-\omega|)
\end{bmatrix}
 \right)=0,
\end{eqnarray*}
from which we have
\begin{equation*}
\lambda^2-(\nu\omega^2  +2|1-\omega|)\lambda +(1-\omega)^2=0.
\end{equation*}
It follows from Lemma~\ref{lem:2.1} that $|\lambda|<1$ (i.e., $\rho(W)<1$) if and only if
\begin{align}
  (1-\omega)^2&<1, \label{eq:con1}\\
  \nu\omega^2  +2|1-\omega|&<1+(1-\omega)^2. \label{eq:con2}
\end{align}
Obviously, the inequality \eqref{eq:con1} holds if and only if $0<\omega<2$. Next, we will continue our discussion by dividing the following two cases.

\textbf{Case 1:} when $0< \omega \leq 1$, the inequality \eqref{eq:con2} becomes
      $$
      \nu\omega^2  +2(1-\omega)<1+(1-\omega)^2 \Leftrightarrow \omega^2 \nu<\omega^2,
      $$
      which holds if $0< \nu<1$.

\textbf{Case 2:} when $1< \omega <2$, the inequality \eqref{eq:con2} becomes
      $$
      \omega^2 \nu +2(\omega-1)<1+(1-\omega)^2 \Leftrightarrow (\nu-1)\omega^2+4\omega-4<0,
      $$
      which holds if $0< \nu< 1$ and $
  1<\omega<\frac{2-2\sqrt{\nu}}{1-\nu}<2.
  $

According to \textbf{Case 1} and \textbf{Case 2}, we can conclude that $\rho(W) < 1$ if \eqref{eq:con-sor} holds.

Finally, if \eqref{eq:con-sor} holds, we can prove the unique solvability of AVE~\eqref{eq:ave}. In contrast, suppose that $\bar{x}_{*}\neq x_*$ is another  solution to AVE~\eqref{eq:ave},  we have
\begin{numcases}{}
\|x_*-\bar{x}_*\|_2 \leq |1-\omega| \|x_*-\bar{x}_*\|_2 +\omega \nu \|y_*-\bar{y}_*\|_2 ,\label{eq:xb1}\\
\|y_*-\bar{y}_*\|_2 \leq|1-\omega| \|y_*-\bar{y}_*\|_2 +\omega \|x_*-\bar{x}_*\|_2,\label{eq:yb1}
\end{numcases}
where $y_{*}=|x_{*}|$ and $\bar{y}_{*}=|\bar{x}_{*}|$.  It follows from \eqref{eq:xb1} and \eqref{eq:yb1} that
\begin{align*}
\|y_*-\bar{y}_*\|_2
&\leq (|1-\omega|+\frac{\omega^2\nu}{1-|1-\omega|})\|y_*-\bar{y}_*\|_2\\
&=\frac{|1-\omega|-(1-\omega)^2+\omega^2\nu}{1-|1-\omega|}\|y_*-\bar{y}_*\|_2.
\end{align*}
Recall \eqref{eq:con2}, we get $\frac{|1-\omega|-(1-\omega)^2+\omega^2\nu}{1-|1-\omega|}<1$, and then
$$\|y_*-\bar{y}_*\|_2 <\|y_*-\bar{y}_*\|_2,$$
which is a contradiction.
\end{proof}

\begin{remark}
The condition \eqref{eq:con-sor} seems simpler than the condition \eqref{eq:cond1} proposed in \cite{kema2017}. The condition \eqref{eq:cond1} proposed in \cite{kema2017} is further investigated in \cite[Theorem 2.2]{chyh2024}. In addition, for given $\nu \in (0,1)$, the following \Cref{fig:sor} demonstrates that the range of $\omega$ determined by \eqref{eq:con-sor} is larger than that giving in \cite[Theorem 2.2]{chyh2024}.

\begin{figure}[htp]
  \centering
  \includegraphics[width=0.7\linewidth]{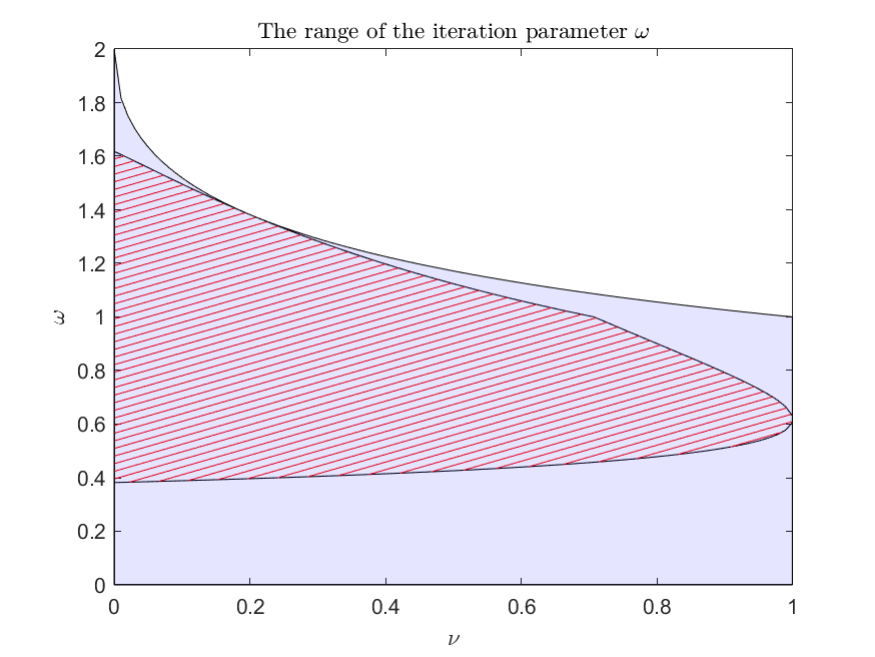}
  \caption{Comparison of convergence domains for the SOR-like method. The light blue area represents the range of $\omega$ obtained from \eqref{eq:con-sor}, and the red striped area represents the range of $\omega$ obtained from \cite[Theorem 2.2]{chyh2024}.}\label{fig:sor}
\end{figure}

\end{remark}

\begin{remark}
The proof of \Cref{thm:sor} can be seen as a new constructive proof of \Cref{pro:us}.
\end{remark}

\subsection{Optimal iterative parameter of SOR-like iteration}
Similar to the derivation of \eqref{eq:W}, we have
\begin{equation}\label{eq:err}
\begin{bmatrix}
  \|x^{(k+1)}-x_*\|_2 \\
  \|y^{(k+1)}-y_*\|_2
\end{bmatrix}
\leq W
\begin{bmatrix}
  \|x^{(k)}-x_*\|_2 \\
  \|y^{(k)}-y_*\|_2
\end{bmatrix}
\le \ldots \le W^{k+1}
\begin{bmatrix}
  \|x^{(0)}-x_*\|_2 \\
  \|y^{(0)}-y_*\|_2
\end{bmatrix}.
\end{equation}
In addition, the small value of $\rho(W)$ is, the faster $\{W^k\}$ will converge to zero later on (as $k\rightarrow +\infty$). Hence, it follows from \eqref{eq:err} that the small value of $\rho(W)$ is, the faster $\{x^{(k)}\}_{k=0}^{\infty}$ will converge to $x_*$ later on. In the following, for given $\nu \in (0,1)$, we will determine the optimal iterative parameter $\omega \in \left(0,\frac{2 - 2\sqrt{\nu}}{1 - \nu}\right)$ by minimizing $\rho(W)$.

Given $\nu \in (0,1)$, for $\omega \in \left(0,\frac{2 - 2\sqrt{\nu}}{1 - \nu}\right)$ we have
\begin{equation*}
\triangle=(\omega^2 \nu +2|1-\omega|)^2-4(1-\omega)^2 > 0,
\end{equation*}
which implies that
\begin{align*}
\rho(W)&=\frac{2|1-\omega|+\omega^2\nu+\sqrt{(2|1-\omega|+\omega^2\nu)^2-4(1-\omega)^2}}{2},\\
&=\frac{2|1-\omega|+\omega^2\nu+\omega\sqrt{4|1-\omega|\nu+\omega^2\nu^2}}{2}.
\end{align*}
Let
\begin{equation*}
g_\nu(\omega)=2|1-\omega|+\omega^2\nu+\omega\sqrt{4|1-\omega|\nu+\omega^2\nu^2},
\end{equation*}
for given $\nu \in (0,1)$, the problem of finding the optimal iterative parameter is changing to find the minimum point of $g_\nu(\omega)$ in $\omega \in \left(0,\frac{2 - 2\sqrt{\nu}}{1 - \nu}\right)$. Then we have the following theorem.

\begin{theorem}\label{thm:op-sor}
Let $A\in \mathbb{R}^{n\times n}$ be a nonsingular matrix and let $\nu=\|A^{-1}\|_2$. Given $\nu \in (0,1)$, the optimal iterative parameter that minimizes $g_\nu(\omega)$ in $\left(0,\frac{2 - 2\sqrt{\nu}}{1 - \nu}\right)$ is $\omega=1$.
\end{theorem}

\begin{proof}
Since
\begin{equation*}
g_\nu(\omega)=
\begin{cases}
  2(1-\omega)+\omega^2\nu+\omega\sqrt{4(1-\omega)\nu+\omega^2\nu^2}, & \text{if}~0<\omega\leq1, \\
  2(\omega-1)+\omega^2\nu+\omega\sqrt{4(\omega-1)\nu+\omega^2\nu^2}, & \text{if}~1<\omega<\frac{-2+2\sqrt{\nu}}{\nu-1},
\end{cases}
\end{equation*}
we have
\begin{equation*}
g^\prime_\nu(\omega)=
\begin{cases}
  -2+2\omega\nu+\frac{\omega\nu(-2+\omega\nu)}{\sqrt{\nu(4-4\omega+\omega^2\nu)}}+\sqrt{\nu(4-4\omega+\omega^2 \nu)}, & \mbox{if}~0<\omega\leq1, \\
  2+2\omega\nu+\frac{\omega\nu(2+\omega\nu)}{\sqrt{\nu(-4+4\omega+\omega^2\nu)}}+\sqrt{\nu(-4+4\omega+\omega^2\nu)}, & \mbox{if}~1<\omega<\frac{-2+2\sqrt{\nu}}{\nu-1}.
\end{cases}
\end{equation*}

When $0<\omega\leq1$, we have
\begin{equation*}
g''_\nu(\omega)=2\nu+\frac{-16\nu^2+12\omega\nu^2+12\omega\nu^3-12\omega^2\nu^3+2\omega^3\nu^4}
{(4\nu-4\nu\omega+\omega^2\nu^2)^{\frac{3}{2}}}
\end{equation*}
and
\begin{equation*}
g'''_\nu(\omega)=-\frac{24(\omega-2)(\nu-1)\sqrt{\nu(\omega^2\nu-4\omega+4)}}
{(\omega^2\nu-4\omega+4)^3}<0.
\end{equation*}
Hence, $g''_\nu$ is monotonically decreasing on the interval $(0, 1]$. Then $g''_\nu(\omega)<0$ with $\omega \in (0, 1]$ since $g''_\nu$ is continuous and $\lim\limits_{\omega\rightarrow 0^{+}} g''_\nu(\omega)=2(\nu-\sqrt{\nu}) < 0$. Thus, $g'_\nu $ is also monotonically decreasing on the interval $(0, 1]$. Similarly, $g'_\nu(\omega)<0$ with $\omega \in (0, 1]$ since $g'_\nu$ is continuous and $\lim\limits_{\omega\rightarrow 0^{+}} g'_\nu(\omega)=2(\sqrt{\nu}-1) < 0$. Hence, $g_\nu $ is monotonically decreasing on the interval $(0, 1]$.

When $1<\omega<\frac{-2+2\sqrt{\nu}}{\nu-1}$, we have $g'_\nu(\omega)>0$ and thus $g_\nu $ is monotonically increasing on the interval $\left(1,\frac{-2+2\sqrt{\nu}}{\nu-1}\right)$.

It follows from the above discussion and the continuity of $g_\nu $ that the minimum point of $g_\nu $ on the interval $\left(0,\frac{-2+2\sqrt{\nu}}{\nu-1}\right)$ is $\omega=1$.
\end{proof}

\begin{remark}
In \cite{chyh2024}, in a different sense, Chen et al. proposed the optimal iterative parameter of the SOR-like iteration of the form
\begin{equation}\label{eq:opt}
\omega^*_{opt}=\begin{cases}
                 \omega_{opt}, & \mbox{if }~\frac{1}{4}<\nu<1, \\
                 1, & \mbox{if}~0<\nu\leq \frac{1}{4},
               \end{cases}
\end{equation}
where $\omega_{opt}\in (0,1)$ is the root of
{\small\begin{align*}
g_{\nu}^1(\omega) &= 6(\omega-1)+8\nu^2\omega^3+2\nu(2\omega-3\omega^2)\\
&\qquad +\frac{[3\left( \omega -1 \right) ^{2}+2\,{\nu}^{2}{\omega}^{4}+2\,\nu{\omega
}^{2} \left( 1-\omega \right)][6(\omega-1)+8\nu^2\omega^3+2\nu(2\omega-3\omega^2)]
-8(\omega-1)^3}{\sqrt{[3\left( \omega -1 \right) ^{2}+2\,{\nu}^{2}{\omega}^{4}+2\,\nu{\omega
}^{2} \left( 1-\omega \right)]^2-4(\omega-1)^4}}.
\end{align*}}
The root of $g_{\nu}^1$ doesn't have a analytical form while it can be approximately calculated by the classical bisection method. Given $\nu\in(0,1)$, our new optimal iterative parameter has a analytical form.
\end{remark}

\section{New convergence and optimal iterative parameter of FPI method}\label{sec:FPI}
In this section, we present new convergence result of  FPI  for solving AVE \eqref{eq:ave} and determine its optimal iterative parameter.

\subsection{New convergence result of FPI}

Similar to the proof of \Cref{thm:sor}, we can obtain the following theorem. However, we remain the sketch of the proof here in order to determine the optimal iterative parameter of  FPI.

\begin{theorem}\label{thm:fpi}
Let $A\in \mathbb{R}^{n\times n}$ be a nonsingular matrix  and $\nu=\|A^{-1}\|_2$. If
\begin{equation}\label{eq:con-fpi}
0< \nu<1 \quad \text{and} \quad 0< \tau <\frac{2}{\nu+1},
\end{equation}
AVE \eqref{eq:ave} has a unique solution for any $b\in \mathbb{R}^n$ and the sequence~$\{(x^{(k)},y^{(k)})\}^\infty_{k=0}$ generated by~\eqref{eq:fpi} globally linearly converges to~$(x_{*}, y_{*}=|x_*|)$, where $|x_*|$ is the unique solution of AVE~\eqref{eq:ave}.
\end{theorem}

\begin{proof}
Similar to the proof of \Cref{thm:sor}, we have
\begin{equation}\label{eq:U}
\begin{bmatrix}
  \|x^{(k+1)}-x^{(k)}\| \\
  \|y^{(k+1)}-y^{(k)}\|
\end{bmatrix}
\leq U
\begin{bmatrix}
  \|x^{(k)}-x^{(k-1)}\| \\
  \|y^{(k)}-y^{(k-1)}\|
\end{bmatrix}
\end{equation}
with
\begin{equation}\label{eq:u}
U=\begin{bmatrix}
    0 & \nu \\
    0 & \tau \nu+|1-\tau|
  \end{bmatrix}\ge 0.
\end{equation}
Then, the proof is completed if $\rho(U)<1$. By some algebra, $\rho(U) < 1$ if \eqref{eq:con-fpi} holds.
\end{proof}

\begin{remark}
\Cref{fig:fpi} illustrates the comparison of convergence domains for FPI, from which we see that our new result substantially extends the convergence domain of \eqref{eq:cfpi}. Moreover, we fill the gap mentioned in \Cref{sec:intro} without modifying the original FPI.

\begin{figure}[htp]
  \centering
  \includegraphics[width=0.7\linewidth]{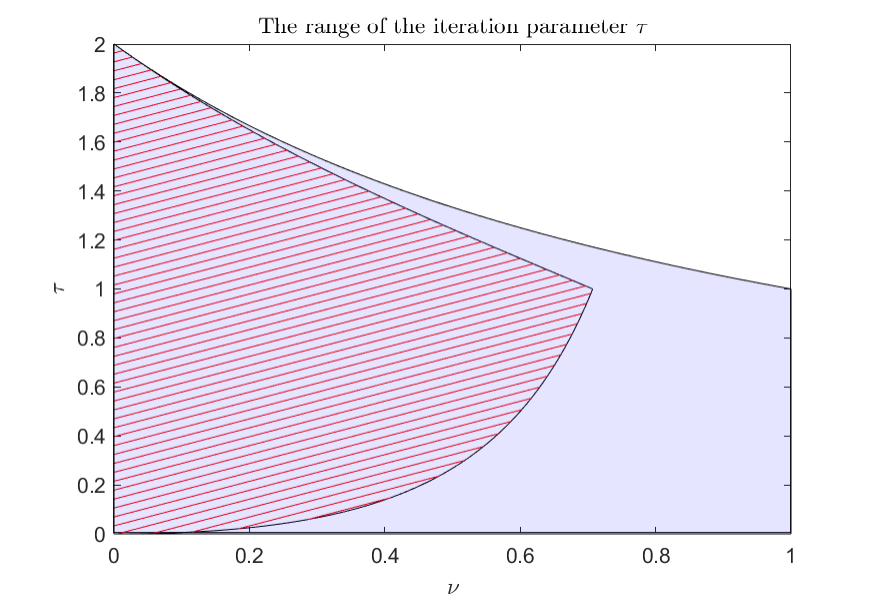}
  \caption{Comparison of convergence domains for the FPI method. The light blue area represents the range of $\tau$ obtained from \eqref{eq:con-fpi}, and the red striped area represents the range of $\tau$ obtained from \eqref{eq:cfpi}. }\label{fig:fpi}
\end{figure}
\end{remark}

\begin{remark}
The proof of \Cref{thm:fpi} can also be seen as a new constructive proof of \Cref{pro:us}.
\end{remark}

\subsection{Optimal iterative parameter of FPI method}
The optimal iterative parameter of FPI is lack in the literature. In this subsection, we will give the optimal iterative parameter which minimizes $\rho(U)$. Similar to the derivation of \eqref{eq:U}, we have
\begin{equation}\label{eq:errfpi}
\begin{bmatrix}
  \|x^{(k+1)}-x_*\|_2 \\
  \|y^{(k+1)}-y_*\|_2
\end{bmatrix}
\leq U
\begin{bmatrix}
  \|x^{(k)}-x_*\|_2 \\
  \|y^{(k)}-y_*\|_2
\end{bmatrix}
\le \ldots \le U^{k+1}
\begin{bmatrix}
  \|x^{(0)}-x_*\|_2 \\
  \|y^{(0)}-y_*\|_2
\end{bmatrix}.
\end{equation}
Hence, it follows from \eqref{eq:errfpi} that the small value of $\rho(U)$ is, the faster $\{x^{(k)}\}_{k=0}^\infty$ will converge to $x_*$ later on. In the following, for given $\nu \in (0,1)$, we will determine the optimal iterative parameter $\tau \in \left(0, \frac{2}{\nu+1}\right)$ that minimizes $\rho(U)$. Specially, we have the following theorem.

\begin{theorem}\label{thm:op-fpi}
Let $A\in \mathbb{R}^{n\times n}$ be a nonsingular matrix and $\nu=\|A^{-1}\|_2$. Given $\nu \in (0,1)$, the optimal iterative parameter that minimizes $\rho(U)$ in $\left(0, \frac{2}{\nu+1}\right)$ is $\tau=1$.
\end{theorem}
\begin{proof}
From \eqref{eq:u}, for given $\nu \in (0,1)$, let
\begin{equation}\label{eq:g}
g_\nu(\tau) = \tau\nu+ |1-\tau|.
\end{equation}

When $0< \tau \leq 1$, \eqref{eq:g} becomes
\begin{equation*}
g_\nu(\tau)=\tau\nu+1-\tau.
\end{equation*}
Then, $g^\prime_\nu(\tau)=\nu-1< 0$. Hence, $g_\nu$ is a monotonically decreasing function in the interval $(0,1]$. When $1<\tau<\frac{2}{\nu+1}$, \eqref{eq:g} becomes
\begin{equation*}
g_\nu(\tau)=\tau\nu+\tau-1,
\end{equation*}
and then we get $g^\prime_\nu(\tau)=\nu+1 > 0$.  Hence, $g_\nu$ is a monotonically decreasing function in the interval $\left(1, \frac{2}{\nu+1}\right)$. In conclusion, in the interval $\left(0, \frac{2}{\nu+1}\right)$, the continuous function $g_\nu$ obtains its minimum value at $\tau = 1$.
\end{proof}

\begin{remark}
An interesting finding is that the SOR-like iteration \eqref{eq:sor} and the FPI \eqref{eq:fpi} are the same when they are equipped with our optimal iterative parameters $\omega = 1$ and $\tau = 1$, respectively.
\end{remark}

\section{Numerical experiments}\label{sec:ne}
In this section, we will present two numerical examples to illustrate the superior performance of the SOR-like iteration method with our optimal iteration parameter ``$\omega_{\rm nopt} = 1$'' (denoted by ``SORLnopt'') for solving AVE \eqref{eq:ave}. We compare the performance of SORLo, SORLopt, SORLaopt, SORLno algorithms mentioned in \cite{chyh2024} and the FPIno algorithm (the FPI method \cite{ke2020} with the numerically optimal iterative parameter ``$\tau_{{\rm no}}$'', which is selected from $\tau = [0.001 : 0.001 : 1.999]$ and is the
first one to reach the minimal number of iteration of the method) with the SORLnopt algorithm.


Note that in all these algorithms, the main task per iteration is solving a system of linear equations. In this paper, the tested methods are implemented in conjunction
with the Cholesky factorization since the coefficient matrix is symmetric positive definite. Specifically, we use $dA = \textbf{\text{decomposition}}(A, \text{`chol'})$ to generate the Cholesky decomposition of $A$, where `decomposition' is the routine in MATLAB, which returns the corresponding decomposition of a matrix $A$ that can be used to solve the linear
system $Ax = b$ efficiently. The call $x = d A \backslash b$ returns the same vector as $A\backslash b$, but is typically faster.

In the numerical results, we will report ``IT" (the number of iteration), ``CPU" (the elapsed CPU time in seconds), and ``RES" (the relative residual error). RES is defined by
$$
\operatorname{RES} =\frac{ \left\| A x^{(k)} - \left| x^{(k)} \right| - b \right\|}{\|b\|}.
$$

All tests are started from the initial zero vectors and terminated if the current iteration satisfies $\operatorname{RES} \leq 10^{-8}$ or the number of prescribed maximal iteration steps $k_{\max} = 100$ is exceeded (denoted by ``-"). We denote ``Range1'' as the range of $\omega$ for the SOR-like iteration determined by (2.24)-(2.26) in \cite{chyh2024}, ``Range2'' as the range of $\omega$ for the SOR-like iteration determined by \eqref{eq:con-sor}, ``Range3'' as the range of $\tau$ for  FPI determined by \eqref{eq:cfpi}, and  ``Range4'' as the range of $\tau$ for FPI determined by \eqref{eq:con-fpi}. In order to obtain more accurate CPU time, we run all test problems in each method five times and take the average. All computations are done in MATLAB R2021a on a personal computer with IntelCore(TM) i7 CPU 2.60 GHz, 16.0GB memory.

\begin{example}\label{exam:5.1}
Consider AVE \eqref{eq:ave} with
\begin{equation*}
A = {\rm Tridiag}(-I_m,S_m,-I_m)=
\begin{bmatrix}
  S_m & -I_m & 0 & \ldots & 0 & 0 \\
  -I_m & S_m & -I_m & \ldots & 0 & 0 \\
  0 & -I_m & S_m & \ldots & 0 & 0 \\
  \vdots & \vdots & \vdots & \ddots & \vdots & \vdots \\
  0 & 0 & 0 & \ldots & S_m & -I_m \\
  0 & 0 & 0 & \ldots & -I_m & S_m
\end{bmatrix}\in\mathbb{R}^{n\times n},
\end{equation*}

\begin{equation*}
S_m={\rm tridiag}(-1,8,-1)=
\begin{bmatrix}
 8 & -1 & 0 & \ldots & 0 & 0 \\
  -1 & 8 & -1 & \ldots & 0 & 0 \\
  0 & -1 & 8 & \ldots & 0 & 0 \\
  \vdots & \vdots & \vdots & \ddots & \vdots & \vdots \\
  0 & 0 & 0 & \ldots & 8 & -1 \\
  0 & 0 & 0 & \ldots & -1 & 8
\end{bmatrix}\in\mathbb{R}^{m\times m}
\end{equation*}
and $b=Ax^*-|x^*|$, where $x^*=[-1,1,-1,1,\cdots,-1,1]^\top\in\mathbb{R}^n$.
Here,we have $n=m^2$.

The parameters for this example are displayed in Table \ref{t1}, from which we can find that$~\omega_{\rm o},~\omega_{{\rm no}},~\omega_{{\rm aopt}},~\omega_{{\rm opt}},~\omega_{{\rm nopt}} \in {\rm Range1} \subset {\rm Range2}$ and $\omega_{{\rm nopt}}, \tau_{\rm no} \in {\rm Range3} \subset {\rm Range4}$. In addition, the larger the value of $\nu$ is, the smaller the range of $\omega$ or $\tau$ is. Numerical results for this example are reported in Table \ref{t2}. From Table \ref{t2}, we find SORLopt, SORLno, SORLnopt and FPIno take the same number of iteration, but SORLnopt is better than others in terms of CPU time. SORLaopt performs the worst in terms of IT, SORLopt always performs better than SORLo and SORLaopt in terms of IT and CPU.

\begin{table}[htp]
\centering
\caption{Parameters for Example~\ref{exam:5.1}}\label{t1}
\begin{tabular}{lllll}
\hline
                                      &$m$ &   &  &  \\ \cline{2-5}
                                      &8                          &16                            &32                          &64\\ \hline
$\nu$                             &0.2358                 &0.2458                      &0.2489                   &0.2497\\
$\omega_{\rm o}$          &1.0671                  &1.0704                     &1.0714                   &1.0717\\
$\omega_{{\rm no}}$      &0.9810                  &0.9830                     &0.9840                   &0.9850\\
$\omega_{{\rm aopt}}$   &0.8354                  &0.8305                     &0.8290                   &0.8286\\
$\omega_{{\rm opt}}$     &1                           &1                              &1                            &1\\
$\tau_{{\rm no}}$           &0.9610                   &0.9660                    &0.9680                   &0.9690 \\
{\rm Range1}                 &(0.3994, 1.3447)   &(0.4003, 1.3347)     &(0.4005, 1.3316)    &(0.4006, 1.3308)\\
{\rm Range2}                &(0, 1.3463)             &(0, 1.3371)              &(0, 1.3343)             &(0, 1.3336) \\
{\rm Range3}                &(0.0369, 1.5956)    &(0.0407, 1.5808)     &(0.0419, 1.5762)    &(0.0422, 1.5750)\\
{\rm Range4}                &(0, 1.6184)             &(0, 1.6054)              &(0, 1.6014)             &(0, 1.6004) \\\hline
\end{tabular}
\end{table}

\begin{table}[htp]
\centering
\caption{Numerical results for Example~\ref{exam:5.1}}\label{t2}
\begin{tabular}{llllll}
\hline
 Method &                  &$m$ &  &  &  \\ \cline{2-6}
  &                             & 8                 & 16                 & 32                 & 64 \\  \hline
 SORLo        &IT       & 16                &16                 &17                  &17                     \\
                    &CPU   &0.0381          &0.0413          &0.0495           &0.1502               \\
                    &RES   &4.9538e-09   &8.5231e-09   &3.3922e-09    &3.6490e-09         \\
 SORLaopt   &IT       &18                 &19                 &19                  &19                      \\
                    &CPU   &0.0365          &0.0402          &0.0493           &0.1463                \\
                   &RES   &8.6677e-09   &4.9002e-09   &5.6512e-09    &5.9552e-09         \\
 SORLopt     &IT       &11                 &11                 &11                  &11                       \\
                    &CPU   &0.0354          &0.0374          &0.0447           &0.1408                \\
                    &RES   &2.6003e-09   &3.2109e-09   &3.5117e-09    &3.6612e-09        \\
 SORLno      &IT       &11                 &11                 &11                  &11                      \\
                    &CPU   &1.0504          &1.7228         &4.4550            &18.1472              \\
                    &RES   &9.5801e-09   &9.8268e-09   &9.8327e-09    &9.4146e-09        \\
 SORLnopt   &IT        &11                 &11                 &11                  &11                     \\
                   &CPU    &0.0088          &0.0103          &0.0102           &0.0143          \\
                   &RES    &2.6003e-09   &3.2109e-09   &3.5117e-09    &3.6612e-09       \\
 FPIno        &IT        &11                 &11                 &11                  &11                     \\
                   &CPU    &0.8285          &1.3509          &3.5003           &14.4657          \\
                   &RES    &9.9594e-09   &9.7292e-09   &9.7047e-09    &9.6682e-09       \\ \hline
\end{tabular}
\end{table}
\end{example}

\begin{example}\label{exam:5.2}
Consider AVE \eqref{eq:ave} with the matrix $A\in\mathbb{R}^{n\times n}$  arises from six different test problems listed in Table \ref{t3}. There matrices are sparse and symmetric positive definite and $\|A^{-1}\|<1$. In addition, let $b=Ax^*-|x^*|$ with $x^*=[-1,1,-1,1,\cdots,-1,1]^\top\in\mathbb{R}^n$.

\begin{table}[htp]
  \centering
  \caption{Problems for \Cref{exam:5.2}}\label{t3}
  \begin{tabular}{lllll}
    \hline
    Problem & $n$ & Problem & $n$ \\\hline
    mesh1e1 & 48 & Trefethen\_{20b} & 19 \\
    mesh1em1 & 48 & Trefethen\_{200b} & 199 \\
    mesh2e1 & 306 & Trefethen\_{20000b} & 19999 \\
    \hline
  \end{tabular}
\end{table}

The parameters for this \Cref{exam:5.2} are displayed in Table \ref{t4}. It follows from Table \ref{t4} that the range of $\omega$ or $\tau$ becomes smaller as the value of $\nu$ becomes larger. Furthermore, all values of $\omega_{{\rm no}},~\omega_{\rm aopt},~\omega_{\rm opt},~\omega_{\rm nopt}\in {\rm Range1},~\omega_{{\rm nopt}}, \tau_{\rm no}\in {\rm Range4}$, while $\omega_{\rm o}$ does not belong to Range1 or Range2 for the former three test problems and it belongs to Range1 for the later three test problems. Moreover, for the test problem mesh2e1, $\tau_{\rm no}$ does not belong to Range3, because $\nu=0.7615>\frac{\sqrt{2}}{2}$ lead to ${\rm Range3}=\emptyset$, while it belongs to Range4. In addition, we can find that~${\rm Range1} \subset {\rm Range2}$ and ${\rm Range3} \subset {\rm Range4}$.

\begin{table}[htp]
\setlength{\tabcolsep}{1pt}
\small
\centering
\caption{Parameters for Example~\ref{exam:5.2}}\label{t4}
\begin{tabular}{lllllll}
 \hline
                                    &Problem              &   & & & &   \\  \cline{2-7}
                                   &mesh1e1              &mesh1em1            &mesh2e1              &Trefethen\_20b      &Trefethen\_200b    &Trefethen\_20000b\\\hline
$\nu$                         &0.5747                  &0.6397                  &  0.7615                &0.4244                  &0.4265                  &0.4268\\
$\omega_{\rm o}$       &1.2105                 &1.2498                 &1.3438                   &1.1372                  &1.1381                  &1.1382\\
$\omega_{\rm no}$     &0.9430                 &0.9290                  &0.9320                  &0.9300                  &0.9510                  &0.9990\\
$\omega_{\rm aopt}$  &0.7102                 &0.6929                   &0.6641                 &0.7569                  &0.7561                  &0.7561\\
$\omega_{\rm opt}$    &0.8218                 &0.7848                  &0.7210                  &0.9114                  &0.9102                  &0.9101 \\
$\tau_{\rm no}$          &0.8920                 &0.8350                  &0.8710                  &0.8690                  &0.8410                  &0.9700\\
{\rm Range1}              &(0.4361, 1.0753)  &(0.4460, 1.0367)   &(0.4692, 0.9413)   &(0.4175, 1.1785)  &(0.4177, 1.1769)   &(0.4177, 1.1767)\\
{\rm Range2}              &(0, 1.1376)           &(0, 1.1112)             &(0, 1.0680)           &(0, 1.2110)            &(0, 1.2099)            &(0, 1.2097)\\
{\rm Range3}              &(0.4271, 1.1547)  &(0.6422, 1.0786)   &$\emptyset $        &(0.1642, 1.3377)   &(0.1665, 1.3351)   &(0.1669, 1.3347)\\
{\rm Range4}              &(0, 1.2701)           &(0, 1.2197)            &(0, 1.1354)           &(0, 1.4041)            &(0, 1.4020)             &(0, 1.4017)\\\hline
\end{tabular}
\end{table}

Table \ref{t5} reports the numerical results for Example \ref{exam:5.2}. From Table \ref{t5}, we can see that for the test problems mesh1e1, mesh1em1, mesh2e1, Trefethen\_20b and Trefethen\_200b, SORLnopt performs better than others in terms of CPU, even though it does not have the fewest number of iterations. For the test problem Trefethen\_20000b, SORLnopt performs the best in terms of IT and CPU. In addition, for problems mesh1em1, Trefethen\_20b and Trefethen\_200b, we can conclude that the number of iteration of SORLopt is less than that of SORLnopt and reverse for other problems. SORLopt is better than SORLo and SORLaopt in terms of IT and CPU. FPIno and SORLno have the longest CPU times compared to other algorithms, and they show the same number of iterations except in problems mesh1em1 and Trefethen\_200b. In summary, for different problems, the optimal parameters of the same algorithm under different metrics may have different performance.

\begin{table}[h]
\setlength{\tabcolsep}{1.5pt}
\centering
\caption{Numerical results for Example~\ref{exam:5.2}}\label{t5}
\begin{tabular}{llllllll}
\hline
  &                           &Method &  &  &  &  \\ \cline{3-8}
 Problem &                                        & SORLo         & SORLaopt    & SORLopt      & SORLno        & SORLnopt      &FPIno          \\  \hline
 mesh1e1                &IT       & --                  &35                 &27                  &19                  &26                   &19                 \\
                               &CPU   & --                  &0.0387          &0.0367           &1.1828           &0.0100           &1.0529           \\
                               &RES   & --                  &7.3969e-09   &5.9299e-09    &9.8415e-09    &8.9665e-09     &9.9665e-09    \\
 mesh1em1            &IT        & --                  &33                 &27                  &18                  &33                   &19                  \\
                               &CPU   & --                  &0.0381          &0.0373           &1.1812           &0.0097            &1.0739           \\
                               &RES   & --                  &8.5320e-09   &5.9776e-09    &9.8672e-09    &8.2026e-09      &9.9313e-09    \\
 mesh2e1                 &IT      & --                  &35                  &31                  &18                  &24                   &18                 \\
                                &CPU  & --                  &0.0289           &0.0280           &2.2952           &0.0103            &2.0373            \\
                                &RES  & --                  &9.0933e-09    &7.3981e-09    &9.6759e-09    &9.3366e-09    &9.9450e-09      \\
 Trefethen\_20b        &IT      &49                 &20                 &13                  &12                  &15                   &12                    \\
                               &CPU   &0.0291          &0.0272          &0.0254           &0.9417           &0.0095            &0.7934             \\
                               &RES   &9.2171e-09   &6.3617e-09   &7.0832e-09    &9.5933e-09    &9.0253e-09     &9.9970e-09      \\
 Trefethen\_200b      &IT      &36                 &14                 &10                  &8                    &11                    &9                   \\
                                &CPU  &0.0478          &0.0435          &0.0423           &3.0857           &0.0112             &2.5563          \\
                                &RES  &8.6534e-09   &9.3612e-09   &3.7711e-09    &9.7946e-09    &7.8160e-09      &9.9102e-09          \\
 Trefethen\_20000b  &IT      &10                 &14                 &8                    &3                   &3                      &3           \\
                                &CPU  &7.9540          &8.2715          &7.8463           &4255.4442      &5.5824            &5416.5100            \\
                                &RES  &6.5615e-09   &2.6424e-09   &4.2785e-09    &7.9680e-09     &8.0323e-09     &9.7652e-09           \\   \hline
\end{tabular}
\end{table}
\end{example}

\section{Conclusions}\label{sec:conclusions}
In this paper, we provide a new convergence analysis of the SOR-like iteration and the FPI method, which result wider convergence range. Based on the new analysis, we obtain the new optimal iterative parameters for both the SOR-like iteration and the FPI, respectively. And we find that with our optimal iterative parameters $\omega=1$ and $\tau=1$, the SOR-like iteration and the FPI are consistent. Two numerical examples are given to demonstrate the superior performance of the SOR-like iteration method (and the FPI method) with our optimal iteration parameters.

\section*{Acknowledgments}
C. Chen was partially supported by the Fujian Alliance of Mathematics
(No. 2023SXLMQN03) and the Natural Science
Foundation of Fujian Province (No. 2021J01661). D. Han was partially supported by the Ministry of Science and Technology of China (No. 2021YFA1003600) and the National Natural Science
Foundation of China (Nos. 12126603 and 12131004).

\bibliographystyle{plain}
\bibliography{sorfpi4ave}

\end{document}